\documentclass{amsart}
\usepackage{mathptmx, amssymb, mathtools, enumerate, colonequals, amscd, url, xcolor}
\definecolor{chianti}{rgb}{0.6,0,0}
\definecolor{meretale}{rgb}{0,0,.6}
\definecolor{leaf}{rgb}{0,.35,0}
\usepackage[colorlinks=true, pagebackref, hyperindex, citecolor=meretale, urlcolor=leaf, linkcolor=chianti]{hyperref}

\newtheorem{theorem}{Theorem}[section]
\newtheorem{lemma}[theorem]{Lemma}
\newtheorem{proposition}[theorem]{Proposition}

\theoremstyle{definition}
\numberwithin{equation}{theorem}
\newtheorem{examplex}[theorem]{Example}
\newenvironment{example}{\pushQED{\qed}\examplex}{\popQED\endexamplex}
\newtheorem{remarkx}[theorem]{Remark}
\newenvironment{remark}{\pushQED{\qed}\remarkx}{\popQED\endremarkx}

\newcommand{\segre}{\,\#\,}
\newcommand{\into}{\lhook\joinrel\longrightarrow}

\DeclareMathOperator{\GL}{GL}
\DeclareMathOperator{\SL}{SL}
\DeclareMathOperator{\divv}{div}
\DeclareMathOperator{\Hom}{Hom}
\DeclareMathOperator{\Proj}{Proj}
\DeclareMathOperator{\Spec}{Spec}
\DeclareMathOperator{\tr}{tr}
\DeclareMathOperator{\Hilb}{Hilb}
\DeclareMathOperator{\chr}{char}
\DeclareMathOperator{\rank}{rank}
\DeclareMathOperator{\sign}{sgn}
\DeclareMathOperator{\Pf}{Pf}
\DeclareMathOperator{\diag}{diag}

\newcommand{\fraka}{\mathfrak{a}}
\newcommand{\frakm}{\mathfrak{m}}
\newcommand{\frakp}{\mathfrak{p}}
\newcommand{\frakq}{\mathfrak{q}}

\newcommand{\CC}{\mathbb{C}}
\newcommand{\NN}{\mathbb{N}}
\newcommand{\PP}{\mathbb{P}}
\newcommand{\QQ}{\mathbb{Q}}
\newcommand{\ZZ}{\mathbb{Z}}
\newcommand{\xx}{\pmb{x}}
\renewcommand{\aa}{\pmb{a}}

\newcommand{\calD}{\mathcal{D}}
\newcommand{\calO}{\mathcal{O}}
\newcommand{\sym}{\mathcal{S}}
\newcommand{\alt}{\mathcal{A}}

\renewcommand{\ge}{\geqslant}
\renewcommand{\le}{\leqslant}
\renewcommand{\to}{\longrightarrow}
\renewcommand{\mod}{\ \operatorname{mod}\,}

\newcommand{\floor}[1]{{\left\lfloor #1 \right\rfloor}}
\newcommand{\ceil}[1]{{\left\lceil #1 \right\rceil}}
\newcommand{\md}[1]{{\left\lvert #1 \right\lvert}}
\newcommand{\deff}[1]{\emph{#1}}

\begin{document}
\title{On the $b$-invariant of a normal graded ring}

\author{Aryaman Maithani}
\address{Department of Mathematics, University of Utah, 155 South 1400 East, Salt Lake City, UT~84112, USA}
\email{maithani@math.utah.edu}

\author{Anurag K. Singh}
\address{Department of Mathematics, University of Utah, 155 South 1400 East, Salt Lake City, UT~84112, USA}
\email{singh@math.utah.edu}

\author{Kei-ichi Watanabe}
\address{Department of Mathematics, College of Humanities and Sciences, Nihon University, Setagaya-ku, Tokyo, 156-8550, Japan}
\email{watnbkei@gmail.com}

\thanks{A.M. was supported by a Simons Dissertation Fellowship, A.M. and A.K.S. by NSF grant DMS~2349623, and K.W. by JSPS Grant-in-Aid for Scientific Research (C)~26~K~06772.}

\dedicatory{Dedicated to the memory of Mitsuyasu Hashimoto}

\begin{abstract}
The $b$-invariant of a normal graded ring was defined recently by Okuma, Watanabe, and Yoshida.
We study this for invariant rings of finite groups, proving that it is nonpositive in the case of abelian groups, but not necessarily otherwise;
more generally, we give a description of the anticanonical module for invariant rings of finite groups.
Turning to infinite groups, we record the anticanonical module and the $b$-invariant for families of determinantal rings.

We also construct standard-graded Cohen-Macaulay rings of characteristic zero, with isolated singularities and negative $b$-invariant, that are not of strongly $F$-regular type, thereby disproving a conjecture of Okuma-Watanabe-Yoshida.
\end{abstract}

\vspace*{-1cm}
\maketitle

\section*{Introduction}

Consider a normal $\NN$-graded ring $R$ that is a finitely generated algebra over its degree zero graded component ${[R]}_0$, that we assume is a field.
Set $\frakm$ to be the homogeneous maximal ideal of $R$.
Following \cite{GW}, the \deff{$a$-invariant} of $R$ is defined as
\begin{equation*}
a(R) \colonequals \max\{i\in\ZZ \mid {[H^{\dim R}_\frakm(R)]}_i\neq 0\}.
\end{equation*}
An equivalent formulation in terms of the graded canonical module $\omega_R$ is
\begin{equation*}
a(R) \colonequals -\min\{i\in\ZZ \mid {[\omega_R]}_i\neq 0\}.
\end{equation*}
The definition has proven to be very useful over the last five decades;
more recently, for a normal $\NN$-graded ring $R$ as above, the \deff{$b$-invariant} of $R$ is defined in \cite{OWY} as
\begin{equation*}
b(R) \colonequals \min\{i\in\ZZ \mid {[\omega_R^{-1}]}_i\neq 0\},
\end{equation*}
where $\omega_R^{-1}\colonequals \Hom_R(\omega_R, R)$ is the \deff{anticanonical module} of the ring $R$.
It is readily seen that $b(R)\ge a(R)$, and that equality holds precisely when $\omega_R$ is a free $R$-module, i.e., precisely when $R$ is quasi-Gorenstein, \cite[Proposition~2.2]{OWY}.

The $b$-invariant turns out to be useful in investigating the \emph{nearly Gorenstein} property;
moreover, for two-dimensional normal rings $R$ of characteristic zero, while $a(R)<0$ implies that $R$ has rational singularities, the stronger condition $b(R)<0$ implies that $R$ is of $F$-regular type, see \cite[Proposition~3.4]{OWY}.
Motivated by this, the authors conjectured that the same remains true in higher dimension, assuming that $R$ has an isolated singularity, \cite[Conjecture~2.3]{OWY}.
We construct counterexamples to this in Section~\ref{section:counterexamples}.

Section~\ref{section:invariants} is devoted to investigating the $b$-invariant for rings of invariants of finite groups.
Consider a finite group $G$ acting on a polynomial ring $S$ by degree-preserving automorphisms;
then $a(S^G)\le -\dim S$, see, for example,~\cite[Proposition~4.1]{GJS}.
When~$G$ is abelian, we prove that $b(S^G)\le 0$, Theorem~\ref{theorem:abelian};
when $G$ is not abelian, we construct families of examples where $b(S^G)>0$.
Section~\ref{section:determinantal} records the calculation of the $b$-invariant for some classical invariant rings, and also for Hankel determinantal rings.

\section{Generalities}
\label{section:generalities}

\subsection*{Fractional ideals}

Let $R$ be a normal $\NN$-graded domain that is a finitely generated algebra over a field ${[R]}_0$.
All modules and isomorphisms that appear in this paper are tacitly assumed to be graded.
Let $K \colonequals Q(R)$ denote the \deff{total homogeneous quotient} of~$R$,
i.e., the~$\ZZ$\nobreakdash-graded ring obtained by inverting nonzero homogeneous elements of $R$.
The ring~$K$ has the property that each nonzero homogeneous element is a unit, so every graded~$K$\nobreakdash-module has a basis.
A (graded) \deff{fractional ideal} of $R$ is a (graded) finitely generated nonzero $R$-submodule $\fraka$ of $K$.
Any finitely generated torsionfree $R$\nobreakdash-module $M$ of rank one is, up to a grading shift, isomorphic to a fractional ideal of $R$, and we define
\begin{equation*}
M^{-1} \colonequals \Hom_R(M, R).
\end{equation*}
Suppose $R \subseteq S$ is a finite extension of graded normal domains.
Then $L \colonequals Q(S)$ may be obtained as the localization of $S$ at homogeneous nonzero elements of $R$.
If $\fraka \subseteq L$ is a finitely generated rank-one $R$-module, then---up to a grading shift---$\fraka$ is isomorphic to a fractional ideal of $R$, and we have an isomorphism of $R$-modules
\begin{align*}
\{f \in L \mid f \fraka \subseteq R\} &\ \cong\ \Hom_R(\fraka, R) \\
f & \ \mapsto\ (a \longmapsto a f).
\end{align*}
In view of this, we shall use the description on the left as the definition of $\fraka^{-1}$ whenever~$\fraka$ is presented as a submodule of an extension $L$.
In particular, if $\fraka$ is contained in $R$, then one has $1 \in \fraka^{-1}$, giving us:

\begin{remark}
\label{remark:if-omega-ideal}
If $\omega_R$ is isomorphic to a graded ideal of $R$, then $b(R) \le 0$.
\end{remark}

\begin{example}
Let $R$ be a normal affine semigroup ring, i.e., a pure subalgebra of a polynomial ring $S\colonequals k[x_1,\dots,x_n]$ generated by monomials.
By Danilov \cite[4.6]{Danilov} and Stanley \cite[Corollary~I.13.1]{Stanley}, $\omega_R$ is isomorphic to a graded ideal of $R$;
see also \cite[Theorem~5.1]{BB}.
Hence $b(R) \le 0$ by Remark~\ref{remark:if-omega-ideal}.
In particular, for $G$ an algebraic torus, the invariant ring~$S^G$ satisfies $\omega_{S^{G}} \into S^{G}$ and consequently $b(S^G) \le 0$;
in Theorem~\ref{theorem:abelian} we prove the same for any finite abelian group.
\end{example}

\begin{example}
\label{example:veronese-canonical}
Fix integers $m,n\ge 2$.
Let $S\colonequals k[x_1,\dots,x_n]$ denote a polynomial ring, and~$R \colonequals S^{(m)}$ its $m$-th Veronese subring.
The ideal $x_1 \cdots x_n S$ is a graded canonical module for $S$; by \cite[Corollary~3.1.3]{GW} one has
\begin{equation*}
\omega_R\ =\ (x_1 \cdots x_n S)\cap R.
\end{equation*}
Its inverse may then be computed as
\begin{equation*}
\omega_R^{-1}\ =\ \{f\in Q(R) \mid f\omega_R \subseteq R\},
\end{equation*}
and this is the fractional ideal of $R$ generated by
\begin{equation*}
\frac{s}{x_1 \cdots x_n}\quad\text{for}\quad s \in{[S]}_{n+m\ZZ}.
\end{equation*}
It follows that $b(R)=-m\floor{n/m}$, while $a(R)= -m\ceil{n/m}$; in particular, $b(R)=0$ if $m>n$.

If $k$ contains a primitive $m$-th root of unity $\zeta$, then $R$ arises as the ring of invariants for the cyclic group $G \colonequals \langle \sigma \rangle$ of order $m$, where $\sigma$ acts on $S$ via $x_i \longmapsto \zeta x_i$.
\end{example}

\begin{remark}
\label{remark:UFD-a-equals-b}
If $R$ is a normal $\NN$-graded domain with ${[R]}_0$ a field, then $\omega_R$ is a reflexive $R$-module of rank one.
Thus, if $R$ is a UFD, then $\omega_R$ is free, and hence $a(R) = b(R)$.
\end{remark}

\subsection*{Normal graded rings and \texorpdfstring{$\QQ$}{Q}-divisors}

Following \cite{Dolgachev, Pinkham, Demazure}, a \deff{$\QQ$-divisor} on a normal projective variety $X$ is a $\QQ$-linear combination of codimension one irreducible subvarieties.
Let $D=\sum n_iV_i$ be such a $\QQ$-divisor, where the subvarieties $V_i$ are distinct.
Set
\begin{equation*}
\floor{D} \colonequals \sum \floor{n_i} V_i,
\end{equation*}
where $\floor{n}$ is the greatest integer less than or equal to $n$, and define
\begin{equation*}
\calO_X(D) \colonequals \calO_X(\floor{D}).
\end{equation*}
With $K(X)$ denoting the field of rational functions on $X$, it follows that
\begin{equation*}
H^0(X,\calO_X(D))\ = \{f\in K(X) \mid \divv(f)+ D\ge 0\}.
\end{equation*}

A $\QQ$-divisor $D$ is \deff{ample} if $ND$ is an ample Cartier divisor for some $N\in\NN$.
In this case, one has the \deff{generalized section ring}
\begin{equation*}
\varGamma_*(X,D)\colonequals\bigoplus_{n\ge0}H^0(X,\calO_X(nD))T^n,
\end{equation*}
where $T$ is an element of degree $1$, transcendental over $K(X)$.
With this notation:

\begin{theorem}{\cite[3.5]{Demazure}}
\label{theorem:Demazure}
Let $R$ be an $\NN$-graded normal domain that is finitely generated over a field ${[R]}_0$.
Suppose $Q(R)$ contains a homogeneous element $T$ of degree $1$.
Then there exists a unique ample $\QQ$-divisor $D$ on $X\colonequals\Proj R$ such that $R=\varGamma_*(X,D)$.
\qed
\end{theorem}

Given a $\QQ$-divisor $D=\sum_i (s_i/t_i)V_i$ with $V_i$ distinct, $s_i$ and $t_i$ being relatively prime integers, and $t_i>0$, set
\begin{equation}
\label{equation:fractional-part}
D'\colonequals\sum_i \frac{t_i-1}{t_i}V_i.
\end{equation}
By \cite[Theorem~2.8]{Watanabe:demazure} and \cite[\S3.2]{Watanabe:dim2}, for each $n\in\ZZ$, one has
\begin{equation*}
{[\omega_R]}_n\ =\ H^0\left(X,\calO_X(K_X+D'+nD)\right)T^n,
\end{equation*}
where $\omega_R$ is the graded canonical module of $R\colonequals\varGamma_*(X,D)$, and $K_X$ is the canonical divisor of $X$;
similarly, the anticanonical module of $R$ has graded components
\begin{equation}
\label{equation:anticanonical}
{[\omega_R^{-1}]}_n\ =\ H^0\left(X,\calO_X(-(K_X+D')+nD)\right)T^n.
\end{equation}

\subsection*{Segre products}

Let $R$, $S$ be $\NN$-graded rings finitely generated over a field~${[R]}_0=k={[S]}_0$.
The \deff{Segre product} of $R$ and $S$ is the $\NN$-graded ring
\begin{equation*}
R\segre S \colonequals \bigoplus_{n\ge 0}{[R]}_n\otimes_k {[S]}_n.
\end{equation*}
We record the K\"unneth formula for local cohomology, \cite[Theorem~4.1.5]{GW}, and some consequences, \cite[Theorems~4.2.3 and~4.3.1]{GW}:

\begin{theorem}[Goto-Watanabe]
\label{theorem:segre}
Let $R$ and $S$ be normal $\NN$-graded rings as above, of dimension at least two.
Set $\frakm_R$, $\frakm_S$, and $\frakm$ to be the homogeneous maximal ideals of the rings~$R$, $S$, and~$R\segre S$ respectively.
Then, for each $d\ge 0$, one has
\begin{equation*}
H^d_{\frakm}(R\segre S)\ =\ \big(R\segre H^d_{\frakm_S}(S)\big) \oplus \big(H^d_{\frakm_R}(R)\segre S\big) \oplus
\bigoplus_{i+j=d+1} \big(H^i_{\frakm_R}(R)\segre H^j_{\frakm_S}(S)\big).
\end{equation*}
In particular, $\dim(R \segre S) = \dim R + \dim S - 1$, and the graded canonical module of $R\segre S$ is
\begin{equation*}
\omega_{R \segre S}\ =\ \omega_R \segre \omega_S.
\end{equation*}
If $R$, $S$ are Cohen-Macaulay with negative $a$-invariants, then $R \segre S$ is Cohen-Macaulay.
\qed
\end{theorem}

\subsection*{Rings of \texorpdfstring{$F$}{F}-regular type}

The theory of tight closure was introduced by Hochster and Huneke in \cite{HH:JAMS}, and further developed in the graded context in~\cite{HH:JAG}.
When~$R$ is an~$\NN$-graded ring that is finitely generated over a field ${[R]}_0$ of positive characteristic, the properties of weak $F$-regularity, $F$-regularity, and strong~$F$-regularity coincide by \cite[Corollary~3.3]{Lyubeznik-Smith}, so we downplay the distinction.

Suppose $R=k[x_1,\dots,x_n]/I$ is a ring finitely generated over a field $k$ of characteristic zero; choose a finitely generated $\ZZ$-algebra $A \subseteq k$ such that
\begin{equation*}
R_A \ =\ A[x_1,\dots,x_n]/I_A
\end{equation*}
is a free $A$-module with $R \cong R_A \otimes_A k$.
The fibers of the homomorphism $A \to R_A$ over maximal ideals $\frakm$ of $A$ are finitely generated algebras over fields of positive characteristic.
We say $R$ has \deff{$F$-regular type} if there exists a finitely generated $\ZZ$-algebra $A \subseteq k$ and a finitely generated $A$-algebra $R_A$ as above such that $R \cong R_A \otimes_A k$, and for all maximal ideals~$\frakm$ in a Zariski dense subset of $\Spec A$, the fiber rings $R_A \otimes_A A/\frakm$ are $F$-regular.

\section{Invariants of finite groups}
\label{section:invariants}

Fix a field $k$ and a finite subgroup $G$ of $\GL_n(k)$.
The group $G$ acts on $S \colonequals k[x_1,\dots,x_n]$ via degree-preserving $k$-algebra automorphisms, where the action of $g \in G$ is given by
\begin{equation*}
\begin{bmatrix} x_1 & \hdots & x_n \end{bmatrix}^{\tr}
\ \longmapsto\
g^{-1} \begin{bmatrix} x_1 & \hdots & x_n \end{bmatrix}^{\tr}.
\end{equation*}
Given a character, i.e., a homomorphism $\mu \colon G \to k^{\times}$, the module of $\mu$-\deff{semiinvariants} is
\begin{equation*}
S^G_{\mu} \colonequals \{s \in S \mid g(s) = \mu(g) s \text{ for all }g \in G\},
\end{equation*}
and we say that $\mu$ is the \deff{eigencharacter} of any nonzero $\mu$-semiinvariant.
It is well-known that $S^G_\mu$ is a reflexive $S^G$-module of rank one, \cite[Lemma~2.1]{Nakajima}.
Set $L \colonequals Q(S)$, and define
\begin{equation*}
{(S^G_{\mu})}^{-1} \colonequals \{f \in L \mid f S^G_{\mu} \subseteq S^G\},
\end{equation*}
as earlier.
While the inclusion $S^G_{\mu^{-1}} \subseteq {(S^G_{\mu})}^{-1}$ is clear, equality holds precisely when $\mu$ is trivial on pseudoreflections, see Remark~\ref{remark:interpret-inverse};
recall that $g\in G$ is a \deff{pseudoreflection} if the matrix $1-g$ has rank one.
A nondiagonalizable pseudoreflection is a \deff{transvection}.
We look at an example next to help calibrate the notation; as $G\subseteq\GL_n(k)$, one has the determinant character $\det\colon G\to k^{\times}$. For $M$ a graded module, $M(i)$ below denotes the module with the shifted grading ${[M(i)]}_j={[M]}_{i+j}$.

\begin{example}
Let $G$ be the subgroup of $\GL_2(\CC)$ generated by
$\sigma \colonequals \begin{bsmallmatrix}\zeta & 0 \\ 0 & \zeta\end{bsmallmatrix}$,
where $\zeta$ is a primitive cube root of unity.
Note that
\begin{equation*}
\sigma\colon x^i y^j\ \longmapsto\ \zeta^{-(i + j)} x^i y^j.
\end{equation*}
As $\det \sigma = \zeta^{-1}$, we obtain the following bases:
\begin{align*}
S^G &\ =\ \CC \{ x^i y^j \mid i + j \equiv 0 \mod 3 \}, \\
S^G_{\det} &\ =\ \CC \{ x^i y^j \mid i + j \equiv 1 \mod 3 \}, \\
S^G_{\det^{-1}} &\ =\ \CC \{ x^i y^j \mid i + j \equiv 2 \mod 3 \},
\end{align*}
giving us
\begin{equation*}
S^G = \CC[x^3,\, x^2 y,\, x y^2,\, y^3], \qquad
S^G_{\det} = S^G\{x,\, y\}, \qquad
S^G_{\det^{-1}} = S^G\{x^2,\, x y,\, y^2\}.
\end{equation*}
Comparing with Example~\ref{example:veronese-canonical}, we see that
$\omega_{S^G} \cong S^G_{\det}(-2)$
and
$\omega_{S^G}^{-1} \cong S^G_{\det^{-1}}(2)$;
these illustrate more general phenomena proven as Theorems~\ref{theorem:canonical-det} and~\ref{theorem:anticanonical-det-inv}.
\end{example}

Recall that the Dedekind different $\calD_{S/S^G}$ is a divisorial ideal of $S$, \cite[\S3.10]{Benson}.
As $S$ is a~UFD, this is a principal ideal generated by a homogeneous semiinvariant $\theta$, whose eigencharacter we denote by $\chi$;
see Remark~\ref{remark:description-theta-f} for a formula for~$\theta$.
The character $\chi$ can be used to detect the quasi-Gorenstein property: the ring $S^G$ is quasi-Gorenstein precisely when the character $\det$ coincides with~$\chi$, see \cite[Corollary~7~(iv)]{Broer}.
The characters $\det$ and~$\chi$ agree on pseudoreflections, as recorded in Lemma~\ref{lemma:det-chi-pseudo}.

The following description of the canonical module is due to Broer \cite[Corollary~7]{Broer}, extending the description for nonmodular actions due to the third author \cite{Watanabe:gorenstein}.

\begin{theorem}[Broer, Watanabe]
\label{theorem:canonical-det}
Let $k$ be a field, $G$ a finite subgroup of $\GL_n(k)$ acting on the polynomial ring $S \colonequals k[x_1,\dots,x_n]$, and let $\theta$ and $\chi$ be as above.
Then the graded canonical module of $S^G$ is
\begin{equation*}
\omega_{S^G} \ \cong\ S^G_{\det/\chi} (-\deg \theta -n),
\end{equation*}
and hence there exists a degree-preserving inclusion
\begin{equation} \label{equation:inclusion-canonical}
\omega_{S^G} \ \into\ S^G_{\det} (-n).
\end{equation}

If $G$ contains no transvections (e.g., the action is nonmodular), then the above is an isomorphism, i.e.,
\begin{equation*}
\pushQED{\qed}
\omega_{S^G} \ \cong\ S^G_{\det} (-n). \qedhere
\popQED
\end{equation*}
\end{theorem}

We give an analogous description of the anticanonical module:

\begin{theorem}
\label{theorem:anticanonical-det-inv}
Let $k$ be a field, $G$ a finite subgroup of $\GL_n(k)$ acting on $S \colonequals k[x_1,\dots,x_n]$, and let $\theta$ and $\chi$ be as above.
The graded anticanonical module of $S^G$ is
\begin{equation*}
\omega_{S^G}^{-1} \ \cong\ S^G_{\chi/\det} (\deg \theta + n),
\end{equation*}
and there is a degree-preserving inclusion
\begin{equation}
\label{equation:inclusion}
S^G_{\det^{-1}}(n) \ \into\ \omega_{S^G}^{-1}.
\end{equation}

If $G$ contains no pseudoreflections, then the above is an isomorphism, i.e.,
\begin{equation*}
\omega_{S^G}^{-1} \ \cong\ S^G_{\det^{-1}}(n).
\end{equation*}

In characteristic zero, the Hilbert series of the anticanonical module---and, in turn, the $b$-invariant---may thus be computed using Molien's formula.
\end{theorem}

\begin{proof}
The descriptions of the anticanonical module follow from Theorems~\ref{theorem:inverses-semiinvariants} and~\ref{theorem:canonical-det}, after noting that $\det/\chi$ is trivial on pseudoreflections, Lemma~\ref{lemma:det-chi-pseudo}. 
The inclusion~\eqref{equation:inclusion} follows from the analogous inclusion~\eqref{equation:inclusion-canonical}, after noting that $S^{G}_{\det^{-1}} \subseteq (S^{G}_{\det})^{-1}$.
\end{proof}

\begin{remark}
Unlike Theorem~\ref{theorem:canonical-det}, where the description
$\omega_{S^G} \cong S^G_{\det}(-n)$
holds in the nonmodular setting, one does not generally have an isomorphism of the form
\begin{equation*}
\omega_{S^G}^{-1} \cong S^G_{\det^{-1}}(d)
\end{equation*}
for $d$ an integer, in the presence of pseudoreflections, even in characteristic zero:
for $\zeta$ a primitive sixth root of unity, consider the cyclic group~$G$ generated by
\begin{equation*}
\sigma \colonequals \begin{bmatrix}\zeta^{-1} & 0 \\ 0 & \zeta^2\end{bmatrix}
\end{equation*}
acting on $S \colonequals \CC[x,y]$.
One checks that $S^G_{\det} = x S^G$, so $S^G$ is Gorenstein with a cyclic anticanonical module, but $S^G_{\det^{-1}}$ is not cyclic.
Note that $\sigma^3$ is a pseudoreflection.
However, such an isomorphism does hold for permutation actions, see Theorem~\ref{theorem:anticanonical-permutation-action}.
\end{remark}

\begin{proposition}
Let $G$ be a finite subgroup of $\GL_n(k)$ acting on $S \colonequals k[x_1,\dots,x_n]$.
\begin{enumerate}[\quad\rm(1)]
\item If $G$ is contained in $\SL_n(k)$, then $b(S^G) \le -n$.
\item Assuming $G$ contains no pseudoreflections, $G \subseteq \SL_n(k)$ if and only if $b(S^G) = -n$.
\end{enumerate}
\end{proposition}

We note that if $G$ contains no pseudoreflections, then the containment $G \subseteq \SL_n(k)$ is equivalent to $S^G$ being quasi-Gorenstein by \cite[Corollary~7~(iv)]{Broer}.

\begin{proof}
If $G$ is contained in $\SL_n(k)$, then~\eqref{equation:inclusion} gives $S^G(n) \into \omega_{S^G}^{-1}$, proving the first statement. If $G$ contains no pseudoreflections, Theorem~\ref{theorem:anticanonical-det-inv} tells us $\omega_{S^G}^{-1} \cong S^G_{\det^{-1}}(n)$.
Thus,
\begin{equation*}
b(S^G) = -n \ \iff\ \left[S^G_{\det^{-1}}\right]_0 \neq 0 \ \iff\ 1 \in S^G_{\det^{-1}},
\end{equation*}
yielding the second statement.
\end{proof}

\begin{remark}
The above may be compared with \cite[Theorem~4.4]{GJS} where it is proven that $a(S^G) = a(S)$ precisely when $G$ is a subgroup of $\SL_n(k)$ that contains no pseudoreflections;
see also \cite{Hashimoto}.
However, the $b$-invariant may remain invariant even otherwise: let $k$ be a field containing a primitive cube root of unity $\zeta$, and consider the group $G$ generated by
\begin{equation*}
\alpha =
\begin{bmatrix}
0 & 1 & 0 \\
1 & 0 & 0 \\
0 & 0 & 1 \\
\end{bmatrix}
\quad \text{and} \quad
\beta =
\begin{bmatrix}
\zeta & 0 & 0 \\
0 & \zeta & 0 \\
0 & 0 & \zeta \\
\end{bmatrix}
\end{equation*}
acting on $S \colonequals k[x,y,z]$.
Then $S^{\langle \alpha \rangle}$ is the polynomial ring $k[x+y,\, xy,\, z]$, while $S^G$ is the third Veronese subring of $S^{\langle \alpha \rangle}$.
Using either Theorem~\ref{theorem:anticanonical-det-inv} or the description of canonical modules for Veronese subrings \cite[Corollary~3.1.3]{GW}, one sees that $b(S^G) = -3 = b(S)$, even though $G$ contains the pseudoreflection~$\alpha$.
If the characteristic of $k$ is not two, then $\alpha$ is not an element of $\SL_3(k)$.
\end{remark}

\begin{remark}
\label{remark:expand-base-field}
Let $\mu \colon G \to k^{\times}$ be a character and $\ell \supseteq k$ a field extension.
Let $S$ be a polynomial ring over $k$, and set $T \colonequals S \otimes_k \ell$ to be the corresponding polynomial ring over~$\ell$.
As $\ell$ is flat over $k$, we conclude that $T^G_{\mu} = S^G_{\mu} \otimes_k \ell$ in view of the exact sequence
\begin{equation*}
\CD
0 @>>> S^G_{\mu}@>>> S @>>> \bigoplus_{g \in G} S,
\endCD
\end{equation*}
where $s\longmapsto \big(g(s) - \mu(g)s\big)_{g \in G}$ under the map to the right.
The characters in Theorem~\ref{theorem:anticanonical-det-inv} are unchanged upon enlarging the base field,
so $\omega_{T^G}^{-1} = \omega_{S^G}^{-1} \otimes_k \ell$ and $b(S^G) = b(T^G)$.
\end{remark}

We next describe the anticanonical module for permutation actions.
The symmetric group $\sym_n$ may be viewed as a subgroup of $\GL_n(k)$, in which case $\det$ restricts to the sign character, that we denote $\sign$.

\begin{theorem}
\label{theorem:anticanonical-permutation-action}
Let $G$ be a subgroup of the symmetric group $\sym_n$, acting on the polynomial ring $S \colonequals k[x_1,\dots,x_n]$ by permuting the indeterminates.
Let $c$ denote the number of transpositions in $G$, and $d$ the minimal degree of a monomial in $S$ whose stabilizer (in $G$) is contained in the alternating group.
\begin{enumerate}[\quad\rm(1)]
\item If $k$ has characteristic two, $S^G$ is quasi-Gorenstein with $a(S^G) = b(S^G) = -(c+n)$.
\item Suppose the characteristic of $k$ differs from two.
Then $\omega_{S^G}^{-1} \cong S^G_{\sign}(2c + n)$,
and thus $b(S^G) = d - 2c - n$, whereas $a(S^G) = -d-n$.
In particular, the $b$-invariant is constant across characteristic not two.
More generally, the Hilbert series $\Hilb(\omega_{S_G}^{-1}, t)$
is also independent of $k$ and may be computed using Molien's formula over $\QQ(t)$ as
\begin{equation*}
\Hilb(\omega_{S_G}^{-1}, t) \ =\ \frac{t^{-2c-n}}{\md{G}}
\sum_{\sigma \in G} \frac{\sign(\sigma)}{\det(1 - \sigma t)}.
\end{equation*}
\end{enumerate}
\end{theorem}

Analogous results for the $a$-invariant and the canonical module for permutation actions were obtained by the first author in \cite{Maithani}.

\begin{proof}
The statement for characteristic two follows from \cite[Corollary~5.2]{Maithani}.
We now assume that $\chr k \neq 2$. In this case $G$ contains no transvections, as the only pseudoreflections in $\sym_n$ are transpositions and these are diagonalizable in characteristic not two.
By Theorem~\ref{theorem:anticanonical-det-inv}, we have
\begin{equation*}
\omega_{S^G}^{-1}\ \cong\ S^G_{\chi/\sign}(\deg \theta + n).
\end{equation*}
By the proof of \cite[Theorem 5.1]{Maithani}, the character $\chi$ takes values $\pm 1$, so $\chi/\sign = \sign/\chi$.
By \cite[Proposition~4.3 and Corollary~4.5]{Maithani}, we have $\deg \theta = c$ and $S^G_{\sign/\chi} \cong S^G_{\sign}(c)$.
Putting these together gives the desired isomorphism
\begin{equation*}
\omega_{S^G}^{-1}\ \cong\ S^G_{\chi/\sign}(\deg \theta + n)\ \cong\ S^G_{\sign}(2c + n).
\end{equation*}
Thus, $b(S^G) = d-2c-n$, where $d$ is the smallest degree of a nonzero element of $S^G_{\sign}$;
that this agrees with the description of $d$ in the theorem is shown in \cite[Corollary~4.8]{Maithani}.
The last statement now follows from the fact that the Hilbert series of $S^G_{\sign}$ is the same across all fields of characteristic other than two, and given by the formula above, see, for example, \cite[Theorem~4.1]{Maithani}.
\end{proof}

\begin{theorem}
\label{theorem:abelian}
Let $S$ be a polynomial ring over a field $k$, and $G$ a finite group acting on~$S$ by degree-preserving $k$-algebra automorphisms.
If $G$ is abelian, then $\omega_{S^G}$ is isomorphic to a graded ideal of $S^G$ and hence $b(S^G)\le 0$.
\end{theorem}

\begin{proof}
It suffices to construct a nonzero $\det^{-1}$-semiinvariant $f$ of degree $n$, for then we obtain an inclusion
\begin{equation*}
\CD
S^{G}_{\det}(-n) @>\cdot f>> S^{G},
\endCD
\end{equation*}
giving us $\omega_{S^{G}} \into S^{G}$ in view of~\eqref{equation:inclusion-canonical}. 
The inequality then follows from Remark~\ref{remark:if-omega-ideal}.

In view of Remark~\ref{remark:expand-base-field}, we may assume that $k$ is algebraically closed. 
Let $p \ge 0$ denote the characteristic of $k$.
As $G$ is abelian, we may write it as a product $P \times H$, where~$P$ is a $p$-group\footnote{The trivial group is considered a $p$-group, especially for $p=0$;
see also \cite{Saenz}.} and $\md{H}$ is invertible in $k$.
Writing $S = k[x_1,\dots,x_n]$, we analyze the action of $G$ on $V \colonequals {[S]}_1$.
The elements of $H$ can be simultaneously diagonalized, resulting in a decomposition
\begin{equation*}
V\ =\ \bigoplus_{\lambda \in H^{\ast}} V_{\lambda},
\end{equation*}
where $H^{\ast} = \Hom(H, k^{\times})$ and
\begin{equation*}
V_{\lambda} \colonequals \{v \in V \mid h(v) = \lambda(h) v \ \text{ for all }\ h \in H\}.
\end{equation*}
Let $\Lambda \subseteq H^{\ast}$ be the set of characters $\lambda$ for which $V_{\lambda} \neq 0$.
As $G$ is abelian, each $V_{\lambda}$ is~$P$\nobreakdash-stable, and, in turn,
$V_{\lambda}^P \neq 0$ for each $\lambda \in \Lambda$, since $P$ is a $p$-group.
Fixing a nonzero element $f_{\lambda} \in V_{\lambda}^P$ for each $\lambda \in \Lambda$, one checks that the product
\begin{equation*}
f \colonequals \prod_{\lambda \in \Lambda} f_{\lambda}^{\dim V_{\lambda}}
\end{equation*}
is a nonzero $\det^{-1}$-semiinvariant of degree $n$;
this may be done separately for elements of~$P$ and of $H$.
\end{proof}

The inequality in Theorem~\ref{theorem:abelian} cannot be improved in general, even for nonmodular cyclic group actions, as we saw in Example~\ref{example:veronese-canonical}.
When $G$ is no longer abelian, $b(S^G)$ may be positive, even in the nonmodular case;
the remainder of this section is devoted to constructing such examples.

\begin{example}
Let $k$ be a field with $\chr k \neq 2$ and consider the action of $G \colonequals \sym_{36}$ on the polynomial ring
\begin{equation*}
S \colonequals k[x_{i,j} \mid 1 \le i \le 3,\ 1 \le j \le 36]
\end{equation*}
in $108$ indeterminates, where $\sigma \in G$ acts via $\sigma(x_{i, j}) = x_{i, \sigma(j)}$.
In other words, the representation is the direct sum of three copies of the natural permutation representation of~$\sym_{36}$;
it follows that the action contains no pseudoreflections.
The action identifies $G$ with a subgroup of $\sym_{108}$ that, in particular, contains no transpositions.
By Theorem~\ref{theorem:anticanonical-permutation-action}, we have $b(S^G) = d - 108$, where $d$ is the minimal degree of a monomial whose stabilizer is contained in the alternating group~$\alt_{108}$.
We shall show that $d = 110$, and hence that $b(S^G) = 2$.

Let $m$ be a monomial of minimal degree whose stabilizer is contained in $\alt_{108}$.
For brevity, write
$m = \xx_1^{\aa_1} \cdots \xx_{36}^{\aa_{36}}$,
where each $\aa_i$ is a triple of nonnegative integers.
If $\aa_i = \aa_j$ for some $i \neq j$, then the element $(i \; j) \in G$ stabilizes $m$;
noting that the image of $(i \; j)$ in $\sym_{108}$ is odd, we conclude that the triples $\aa_i$ are pairwise distinct.
For $s \ge 0$, there are $\binom{s+2}{2}$ distinct triples that sum up to $s$.
As $\sum_{s=0}^4 \binom{s+2}{2} = 35$,
we see that the minimal possible degree is obtained by picking all distinct triples with sum at most $4$, and one triple with sum equal to $5$.
This gives
\begin{equation*}
\deg m \ =\
0 \cdot \binom{2}{2} +
1 \cdot \binom{3}{2} +
2 \cdot \binom{4}{2} +
3 \cdot \binom{5}{2} +
4 \cdot \binom{6}{2} +
5 \cdot 1 \ =\ 110. \qedhere
\end{equation*}
\end{example}

A rich source of two-dimensional rings with $b(S^G)>0$ comes from $\QQ$-divisors on $\PP^1$.

\begin{theorem}
\label{theorem:effective}
Let $V_1$, $V_2$, $V_3$ be distinct points of $\PP^1$ over $\CC$, and consider a triple of positive integers $(t_1,t_2,t_3)$ that is~$(2,2,m)$ for $m\ge 2$, or $(2,3,m)$ for $3\le m\le 5$.
Set
\begin{equation*}
D \colonequals \frac{s_1}{t_1}V_1 + \frac{s_2}{t_2}V_2 + \frac{s_3}{t_3}V_3,
\end{equation*}
where $s_i$ is a positive integer coprime to $t_i$. Then the ring $\varGamma_*(\PP^1,D)$ arises as an invariant ring $S^G$ for a finite group $G$ acting on a polynomial ring $S$, and satisfies $b(S^G)>0$.
\end{theorem}

The fraction field of $\varGamma_*(\PP^1,D)$ has homogeneous elements of degree one, though that is typically not the case with the invariant ring $S^G$; the gradings need to be rescaled, with the scaling factor being $\left(\md{G} \deg D\right)^{1/2}$, as computed in Remark~\ref{remark:rescale}.

\begin{proof}
Recall the definition of $D'$ from~\eqref{equation:fractional-part};
the triples $(t_1,t_2,t_3)$ are precisely those for which $\deg D'<2$, and hence precisely those for which $\varGamma_*(\PP^1,D)$ has log-terminal singularities, see \cite[\S4]{Watanabe:dim2}.
Since the ring $\varGamma_*(\PP^1,D)$ has dimension two, it follows that it is indeed an invariant ring: two-dimensional log terminal singularities are quotient singularities; the local analytic statement is proven as \cite[Proposition~4.18]{Kollar:Mori}, see also \cite{Mehta:Srinivas}, and the graded case follows using~\cite{Hara}.

Working up to linear equivalence, take $K_{\PP^1}=-V_1-V_2$, so that
\begin{equation*}
-(K_{\PP^1}+D')\ =\ \frac{1}{t_1}V_1 + \frac{1}{t_2}V_2 - \frac{t_3-1}{t_3}V_3.
\end{equation*}
Since $D$ has nonnegative coefficients, it follows from~\eqref{equation:anticanonical} that ${[\omega_{S^G}^{-1}]}_n=0$ for $n\le 0$, and hence that $b(S^G)>0$.
\end{proof}

\begin{remark}
\label{remark:rescale}
Let $R$ be an $\NN$-graded domain that is finitely generated over a field ${[R]}_0$.
Following \cite[\S2.4]{Benson}, the \deff{degree} of $R$ is the rational number
\begin{equation*}
\deg R \colonequals \lim_{t\rightarrow 1}\,(1-t)^{\dim R} \Hilb(R, t).
\end{equation*}
When $R\subseteq S$ is a module-finite extension of $\NN$-graded domains that are finitely generated over ${[S]}_0={[R]}_0$,
it is readily seen that
\begin{equation*}
\deg S\ =\ (\deg R)(\rank_R S),
\end{equation*}
see, for example, \cite[Lemma~2.4.1]{Benson}.
It follows that for a finite subgroup $G$ of $\GL_n(k)$ acting on a polynomial ring $S$, one has
\begin{equation*}
\deg S^G\ =\ \frac{1}{\md{G}}.
\end{equation*}
On the other hand, given a generalized section ring $\varGamma_*(X,D)$ as in Theorem~\ref{theorem:Demazure},
\[
\deg \varGamma_*(X,D)\ =\ \deg(D^{\dim X})
\]
by \cite[Proposition~2.1]{Tomari}. Putting these together for invariant rings of dimension two, it follows that the rescaling factor between $\varGamma_*(\PP^1,D)$ and $S^G$ is $\left(\md{G} \deg D\right)^{1/2}$.
For more on $\QQ$-divisors and irredundant gradings, we mention~\cite{STW}.
\end{remark}

We record a specific case of the above construction:

\begin{example}
\label{example:order-24-a}
Set $\PP^1\colonequals\Proj \CC[u,v]$, with points parametrized by $u/v$, and
\begin{equation*}
D \colonequals \frac{1}{2}(2) + \frac{1}{2}(-2) + \frac{1}{2}(\infty).
\end{equation*}
Then $\varGamma_*(\PP^1,D)$ is the $\CC$-algebra generated by
\begin{equation*}
T, \quad
\frac{u^3T^2}{v(u^2-4v^2)}, \quad
\frac{u^2vT^2}{v(u^2-4v^2)}, \quad
\frac{uv^2T^2}{v(u^2-4v^2)}, \quad
\frac{v^3T^2}{v(u^2-4v^2)},
\end{equation*}
where $T$ is defined to be an element of degree one that is transcendental over $\CC(u/v)$.
A straightforward calculation using~\eqref{equation:anticanonical} shows that $b(\varGamma_*(\PP^1,D))=1$.
Note that $\varGamma_*(\PP^1,D)$ may be viewed as the subring of the hypersurface
\begin{equation}
\label{equation:hypersurface}
\CC[T,u,v]/(T^2-v(u^2-4v^2))
\end{equation}
generated as a $\CC$-algebra by $T$, $u^3$, $u^2v$, $uv^2$, $v^3$.

To realize $\varGamma_*(\PP^1,D)$ as an invariant ring, let $G$ be the subgroup of $\GL_2(\CC)$ generated by the matrices
\begin{equation*}
\alpha =
\begin{bmatrix}
\iota & 0 \\
0 & -\iota
\end{bmatrix},
\qquad
\beta =
\begin{bmatrix}
0 & -1 \\
1 & 0
\end{bmatrix},
\qquad
\gamma =
\begin{bmatrix}
\zeta & 0\\
0 & \zeta
\end{bmatrix},
\end{equation*}
where $\iota^2=-1$ and $\zeta$ is a primitive cube root of unity.
It is readily seen that $\md{G}=24$. The invariant ring for the action of $\langle \alpha, \beta \rangle$ on $S \colonequals \CC[x,y]$ is
\begin{equation*}
S^{\langle \alpha, \beta \rangle}\ =\ \CC\big[xy(x^4 - y^4), \ \ x^4 + y^4, \ \ x^2 y^2\big],
\end{equation*}
while $S^G$ is its third Veronese subring
\begin{equation*}
S^G \ =\ \CC\big[xy(x^4 - y^4), \ \ (x^4 + y^4)^3, \ \ x^2y^2(x^4 + y^4)^2, \ \ x^4y^4(x^4 + y^4), \ \ x^6y^6\big],
\end{equation*}
where, admittedly, one of the generators for $S^G$ is redundant.
Setting
\begin{equation*}
T\colonequals xy(x^4 - y^4),\quad u\colonequals x^4 + y^4,\quad v\colonequals x^2y^2,
\end{equation*}
identifies $S^G$ with the subring of the hypersurface~\eqref{equation:hypersurface} generated by $T$, $u^3$, $u^2v$, $uv^2$, $v^3$, i.e., with the ring $\varGamma_*(\PP^1,D)$.
Note that $\varGamma_*(\PP^1,D)$ is generated in degrees $1$ and $2$, while~$S^G$, as a subring of the polynomial ring $S$, is generated in degrees $6$ and $12$;
bearing this in mind, our earlier calculation gives $b(S^G) = 6$,
in accordance with $(\md{G} \cdot \deg D)^{1/2} = 6$ as noted in Remark~\ref{remark:rescale}.
The~$b$-invariant may alternatively be computed as follows:

As $G$ contains no pseudoreflections, we have $\omega_{S^G}^{-1} \cong S^G_{\det^{-1}}(2)$ by Theorem~\ref{theorem:anticanonical-det-inv}.
As $\alpha$ and~$\beta$ have determinant $1$, one looks at the action of $\gamma$ on $S^{\langle \alpha, \beta \rangle}$ and sees that $x^4 y^4$ is a $\det^{-1}$-semiinvariant of least degree;
it follows that $b(S^G) = 8 - 2 = 6$.
More generally, Molien's formula gives
\begin{equation*}
\Hilb(\omega_{S^G}^{-1}, t) \ =\ t^{-2}\Hilb(S^G_{\det^{-1}}, t) \ =\ \frac{3 t^6}{(1 - t^6)(1 - t^{12})}.
\end{equation*}

Similar calculations show that if $\zeta$, in the definition of $G$, is replaced by a primitive~$n$\nobreakdash-th root of unity with $n \ge 3$ odd, then $\md{G} = 8n$ and $b(S^G)=2n$, while $a(S^G)=-2n$.
In particular, $b(S^G)$ can be arbitrarily large, even in a fixed dimension.
\end{example}

Working through the two-dimensional representations of groups of small order using \texttt{Magma} \cite{Magma}, $24$ is the smallest order of a subgroup $G$ of $\GL_2(\CC)$ for which $b(S^G)>0$. Up to isomorphism, there are two such invariant rings with positive $b$-invariant; the other invariant ring is recorded next.

\begin{example}
Let $G$ be the subgroup of $\GL_2(\CC)$ generated by the matrices
\begin{equation*}
\alpha =
\begin{bmatrix}
\zeta & 0 \\
0 & \zeta^{-1}
\end{bmatrix}
\quad \text{and} \quad
\beta =
\begin{bmatrix}
0 & \iota \\
1 & 0
\end{bmatrix},
\end{equation*}
where, once again, $\iota^2=-1$ and $\zeta$ is a primitive cube root of unity.
Note that $\alpha$ has order~$3$ while $\beta$ has order $8$.
The element $\alpha$ generates a normal subgroup, with $G$ being the semidirect product 
$\langle \alpha \rangle \rtimes \langle \beta \rangle$.
It is readily seen that the invariant ring for the action of~$G$ on $S \colonequals \CC[x,y]$ is generated by
\begin{equation}
\label{quation:generators-b}
xy(x^6-\iota y^6),\quad
x^4y^4,\quad
x^{12} + y^{12},\quad
x^3y^3(x^6+\iota y^6).
\end{equation}
Once again, the group $G$ contains no pseudoreflections, so we have $\omega_{S^G}^{-1} \cong S^G_{\det^{-1}}(2)$ by Theorem~\ref{theorem:anticanonical-det-inv}. 
As $\det \alpha = 1$, any $\det^{-1}$-semiinvariant is fixed by $\alpha$ and hence is an element of the subring $S^{\langle \alpha \rangle} = \CC[x^3,\, xy,\, y^3]$. 
Considering the action of $\beta$ on $S^{\langle \alpha \rangle}$, it is readily seen that there is no $\det^{-1}$-semiinvariant of degree less than three, giving us $b(S^G) \ge 1$. 
One may further check that $x^3y^3$ is a $\det^{-1}$-semiinvariant of least degree, so $b(S^G)=6-2=4$. 

To realize the ring $S^G$ in terms of a $\QQ$-divisor, consider $\PP^1\colonequals\Proj \CC[u,v]$ with points parametrized by $u/v$, set $\theta\colonequals \sqrt{-4\iota}$ in $\CC$ and
\begin{equation*}
D \colonequals \frac{1}{2}(\theta) + \frac{1}{2}(-\theta) - \frac{1}{3}(\infty).
\end{equation*}
Then $\varGamma_*(\PP^1,D)$ is the $\CC$-algebra generated by
\begin{equation*}
\frac{uvT^2}{u^2-\theta^2v^2}, \quad
\frac{v^2T^2}{u^2-\theta^2v^2}, \quad
\frac{uvT^3}{u^2-\theta^2v^2}, \quad
\frac{v^2T^3}{u^2-\theta^2v^2},
\end{equation*}
where $T$ is an element of degree one that is transcendental over $\CC(u/v)$.
Choosing
\begin{equation*}
T \colonequals \left(\frac{u^2-\theta^2v^2}{v}\right)^{1/2}
\end{equation*}
the generators then are $u$, $v$, $uT$, $vT$, which may be identified respectively with~\eqref{quation:generators-b}.
Using~\eqref{equation:anticanonical}, one sees that $b(\varGamma_*(\PP^1,D))=1$, from which it follows that $b(S^G) = 4$.
\end{example}

\section{Counterexamples: rings that are not of \texorpdfstring{$F$}{F}-regular type}
\label{section:counterexamples}

We construct counterexamples to \cite[Conjecture~2.3]{OWY}, specifically, we construct $\NN$-graded rings $R$, with ${[R]}_0$ a field of characteristic zero, such that $R$ is normal, with an isolated singular point, $b(R)<0$, though $R$ is not of $F$-regular type.
Such rings~$R$ exist in each dimension greater than or equal to three;
if we require furthermore that $R$ is Cohen-Macaulay, then such examples exist in each dimension greater than or equal to four.
Our constructions in Theorem~\ref{theorem:counterexamples} rely upon the following:

\begin{lemma}
\label{lemma:counterexample}
Let $A$ be a standard-graded normal ring, with $\dim A\ge2$, that is finitely generated over a field ${[A]}_0$. Assume that $A$ has an isolated singularity, and that $\omega_A$ is isomorphic to a graded ideal of $A$; the latter implies, in particular, that $a(A)\le 0$.

Then there exists an $\NN$-graded normal ring $R$, finitely generated over ${[R]}_0 = {[A]}_0$, such that the ring $R$ has an isolated singularity, $b(R)<0$, and $R$ contains an isomorphic copy of $A$ as a graded pure subring. If $A$ is Cohen-Macaulay with $a(A)<0$, we may take $R$ to be Cohen-Macaulay as well.

There exists a ring $R$ as above with $\dim R=\dim A+1$; if one requires, in addition, that~$R$ is standard-graded, then there exists $R$ with $\dim R=\dim A+2$.
\end{lemma}

\begin{proof}
Set $k\colonequals {[A]}_0$ and let $z_1,\dots,z_n$ denote a $k$-basis for ${[A]}_1$.
Set $B$ to be the standard-graded polynomial ring $k[u,v]$, and set $R$ to be the Segre product $A\segre B$,
i.e., $R$ is the $k$-subalgebra of $A\otimes_k B$ generated by elements of the form
\begin{equation}
\label{equation:algebra-generators}
z_iu, \quad z_iv.
\end{equation}
Since $A\otimes_k B$ is normal, so is its pure subring~$R$.
Note that $\dim R=\dim A+1$.
Fix graded canonical modules $\omega_A\subseteq A$ and $\omega_B\colonequals uvB$.
By Theorem~\ref{theorem:segre}, a graded canonical module for $R$ may then be constructed as
\begin{equation*}
\omega_R\ =\ \omega_A \segre \omega_B;
\end{equation*}
this is the $k$-span of elements
\begin{equation}
\label{equation:canonical-basis}
au^{m-i}v^i
\end{equation}
where $a \in {[\omega_A]}_m$ for $m \ge \max\{-a(A),\, 2\}$, and $1\le i\le m-1$.
Note that the $k$-dual of $\omega_R$, namely $H^{\dim R}_{\frakm_R}(R)$, is the $k$-span of
\begin{equation}
\label{equation:lc-basis}
\alpha[1/u^{m-i}v^i],
\end{equation}
where $\alpha \in {[H^{\dim A}_{\frakm_A}(A)]}_{-m}$ for $m$ and $i$ as above.

We now \emph{regrade} $R$ as follows:
take the $\ZZ$-grading on $A\otimes_k B$ where $A$ retains its original grading, while $\deg u=0$ and $\deg v=1$.
The algebra generators~\eqref{equation:algebra-generators} for $R$ then have
\begin{equation}
\label{equation:regrading}
\deg z_iu=1\ \text{ and }\ \deg z_iv=2.
\end{equation}
The highest degree of a nonzero element~\eqref{equation:lc-basis} is now $-1-\max\{-a(A),\, 2\}$,
while the least degree of a nonzero element~\eqref{equation:canonical-basis} is $1+\max\{-a(A),\, 2\}$.
It follows that the elements~\eqref{equation:canonical-basis} continue to span the \emph{graded} canonical module $\omega_R$ for $R$ under the new grading.
While we consider multiple gradings through the course of this proof, it is the grading~\eqref{equation:regrading} that yields the ring $R$, as in the statement of the lemma, with $\dim R=\dim A+1$.

Having determined the graded canonical module $\omega_R$, it is readily seen that its inverse
\begin{equation*}
\omega_R^{-1}\ =\ \{f\in Q(R) \mid f\omega_R \subseteq R\}
\end{equation*}
contains the element
\begin{equation*}
\frac{u}{v}\ =\ \frac{z_1u}{z_1v},
\end{equation*}
since for an element $au^{m-i}v^i \in \omega_R$ as in~\eqref{equation:canonical-basis}, one has
\begin{equation*}
\frac{u}{v}\cdot au^{m-i}v^i\ =\ au^{m-i+1}v^{i-1} \in\ R.
\end{equation*}
As $\deg(u/v)=-1$, we conclude that $b(R)<0$.

We next verify that $R$ has an isolated singularity; inverting a typical generator $z_1u$, it suffices to check that the ring $R_{z_1u}$ is regular. Since
\begin{equation*}
{\left[A_{z_1}\right]}_0\ =\ k\Big[\frac{z_2}{z_1},\ \dots,\ \frac{z_n}{z_1}\Big]
\end{equation*}
is regular, so is the localization
\begin{equation*}
R\Big[\frac{1}{z_1u}\Big]
\ =\ {\left[A_{z_1}\right]}_0\Big[z_1u,\ \frac{1}{z_1u},\ \frac{v}{u}\Big].
\end{equation*}

To see that $R$ contains an isomorphic copy of $A$ as a pure subring, consider the grading on $A\otimes_k B$ where $A$ once again retains its original grading, while
\begin{equation*}
\deg u=-1\ \text{ and }\ \deg v=0.
\end{equation*}
Then $R$ is $\NN$-graded, and its pure subring
${[R]}_0\ =\ k[z_1u,\dots,z_nu]$
is isomorphic to $A$.
Note that the grading that $A\cong k[z_1u,\dots,z_nu]$ inherits as a subring of $R$, using the grading~\eqref{equation:regrading} on the latter, agrees with the standard grading on $A$.

Lastly, if $A$ is Cohen-Macaulay with $a(A)<0$, Theorem~\ref{theorem:segre} shows that $R$ is Cohen-Macaulay.
This completes the construction of the $\NN$-graded ring $R$ with $\dim R=\dim A+1$.

We now sketch the construction of a \emph{standard-graded} ring $R$ with $\dim R=\dim A+2$.
Consider the polynomial ring $B\colonequals k[u,v,w]$ with $\deg u=1=\deg v$ and $\deg w=2$, and set~$R$ to be the Segre product $A\segre B$, i.e., $R$ is the $k$-subalgebra of $A\otimes_k B$ generated by
\begin{equation}
\label{equation:standard-algebra-generators}
z_iu, \quad z_iv, \quad z_iz_jw.
\end{equation}
Then $R$ is a normal ring with $\dim R=\dim A+2$.
Fix graded canonical modules $\omega_A\subseteq A$ and $\omega_B\colonequals uvwB$ for $B$, in which case 
\begin{equation*}
\omega_R\ =\ \omega_A \segre \omega_B
\end{equation*}
is a graded canonical module for $R$; there exists a basis for $\omega_R$ using elements of the form
\begin{equation*}
auvwb
\end{equation*}
where $a \in {[\omega_A]}_m$ and $b \in {[B]}_{m-4}$.

We now regrade $R$ using the $\ZZ$-grading on $A\otimes_k B$ where $A$ retains its original grading, while $\deg u=0=\deg v$ and $\deg w=-1$.
The algebra generators~\eqref{equation:standard-algebra-generators} for $R$ then have degree $1$, so $R$ is standard-graded;
one verifies as before that $\omega_R$ is a graded canonical module under this new grading.
The element
\begin{equation*}
\frac{w}{uv}\ =\ \frac{z_1^2w}{z_1u\cdot z_1v},
\end{equation*}
belongs to $\omega_R^{-1}$
since
\begin{equation*}
\frac{w}{uv}\cdot auvwb \ =\ a w^2b\ \in\ R.
\end{equation*}
As $\deg(w/uv)=-1$, one has $b(R)<0$.

The verification that $R$ has an isolated singularity is similar, using, for example, that
\begin{equation*}
R_{z_1u}\ =\ {\left[A_{z_1}\right]}_0\Big[z_1u,\ \frac{1}{z_1u},\ \frac{v}{u},\ \frac{w}{u^2}\Big]
\qquad\text{and}\qquad
R_{z_1^2w}\ =\ {\left[A_{z_1}\right]}_0\Big[z_1^2w,\ \frac{1}{z_1^2w},\ \frac{u}{z_1w},\ \frac{v}{z_1w}\Big],
\end{equation*}
are regular, each being obtained from ${\left[A_{z_1}\right]}_0$ by adjoining three algebraically independent elements, and inverting one of those.

To see that $R$ contains an isomorphic copy of $A$ as a pure subring, consider the grading on $A\otimes_k B$ where $A$ once again retains its original grading, while
\begin{equation*}
\deg u=-1, \quad \deg v=0=\deg w.
\end{equation*}
Then $R$ is $\NN$-graded, and its pure subring
${[R]}_0\ =\ k[z_1u,\dots,z_nu]$
is isomorphic to $A$. As before, if $A$ is Cohen-Macaulay with $a(A)<0$, then $R$ is Cohen-Macaulay.
\end{proof}

\begin{theorem}
\label{theorem:counterexamples}
For each integer $d\ge 3$, there exists an $\NN$-graded ring $R$ of dimension~$d$, finitely generated over a field ${[R]}_0$ of characteristic zero, such that $R$ is normal, with an isolated singular point, $b(R)<0$, and such that $R$ is not of $F$-regular type.
\begin{enumerate}[\quad\rm(1)]
\item If $d=4$, there exist such rings $R$ that are standard-graded.
\item If $d\ge 4$, there exist such rings $R$ that are Cohen-Macaulay.
\item If $d\ge 5$, there exist such rings $R$ that are standard-graded and Cohen-Macaulay.
\end{enumerate}
\end{theorem}

\begin{proof}
Take $k$ to be a field of characteristic zero, and set $A\colonequals k[z_1,z_2,z_3]/(z_1^3+z_2^3+z_3^3)$, and use Lemma~\ref{lemma:counterexample}.
The resulting ring $R$ contains an isomorphic copy of $A$ as a pure subring, and this remains true when $k$ is replaced by a field of positive characteristic; hence $R$ does not have $F$-regular type.
This settles the case where $d=3$, and also (1) above.

By \cite[Theorem~3.1]{KSSW}, there exist standard-graded Cohen-Macaulay rings $A$ of each dimension three and higher, with $a(A)<0$ and $\omega_A \subseteq A$ a graded ideal, such that~$A$ has an isolated singularity but is not of $F$-regular type.
Using Lemma~\ref{lemma:counterexample} once again, we obtain rings $R$ settling the remaining cases.
For the convenience of the reader, a brief sketch of the construction of $A$ is included next;
see also \cite[Example~5.4.1]{Singh:thesis} or \cite[Example~7.3]{Singh:veronese}.
\end{proof}

\begin{remark}
\label{remark:kssw}
Let $k$ be a field of characteristic zero.
Fix $m\ge 3$, and consider the polynomial ring $k[x_1,\dots,x_m,y_1,y_2]$ with the $\NN^2$-grading where $\deg x_i=(1,0)$, and $\deg y_j=(0,1)$.
Let~$f$ be a homogeneous general polynomial of degree $(m,1)$;
if $m=3$, one may take
\begin{equation*}
f\ =\ x_1x_2x_3y_1-(x_1^3+x_2^3+x_3^3)y_2,
\end{equation*}
while, in general, for $m\ge 3$ one may take
\begin{equation*}
f\ =\ \sum_{i=1}^m x_i^m \ell_i,
\end{equation*}
where $\ell_1,\dots,\ell_m$ are general linear forms in the indeterminates $y_1$ and $y_2$,
i.e., such that~$\ell_i,\ell_j$ are linearly independent for all $i\neq j$.
Set
\begin{equation*}
S\colonequals k[x_1,\dots,x_m,y_1,y_2]/(f),
\end{equation*}
and consider the \deff{diagonal subalgebra} $A\colonequals {[S]}_\Delta$, where $\Delta$ is the diagonal $\{(i,i)\ |\ i\in\ZZ\}$ in~$\ZZ^2$.
The local cohomology and canonical module of $A$ are computed in \cite{KSSW};
the key point is that restricting to graded components is an exact functor.
Using this, it follows that $A$ is Cohen-Macaulay with graded canonical module ${[S(0,-1)]}_\Delta$, and that $A$ is not of $F$-regular type.
The hypothesis that $f$ is general (of the specified degree) ensures that $A$ has an isolated singular point.

The ring $A$ admits a standard grading where ${[A]}_i={[S]}_{(i,i)}$.
With this grading,
\begin{equation*}
{[S(0,-1)]}_\Delta\ =\ \sum x_iA
\end{equation*}
is generated in degree $1$, so that $a(A)=-1$. The ideal $(x_1y_1,\dots,x_my_1)A$ serves as a graded canonical module for $A$.
\end{remark}

\section{Determinantal rings}
\label{section:determinantal}

We record the $b$-invariant, and, for comparison, the $a$-invariant, for classical invariant rings, and also for Hankel determinantal rings;
these are largely straightforward consequences of results existing in the literature. 
Given a matrix $X$, we let $\frakp_{t}(X)$ (resp.~$\frakq_{t}(X)$) denote the ideal generated by the maximal minors of the first $t$ columns (resp. rows) of~$X$.

\begin{theorem}
\label{theorem:determinantal}
Let $k$ be a field.  
\begin{enumerate}[\quad\rm(1)]
\item Let $X$ be an $m\times n$ matrix of indeterminates, and $R\colonequals k[X]/I_{t+1}(X)$, where $I_{t+1}(X)$ is the ideal generated by the size~$t+1$ minors of $X$. Assume that $2 \le t+1 \le m \le n$.
Let $e_1 \le \cdots \le e_m$ and $f_1 \le \cdots \le f_n$ be sequences of integers with $e_i + f_j$ positive for each $i$, $j$, and set $\deg X_{ij} \colonequals e_i + f_j$.
Let $d$ denote the least degree of a $t$-minor.
Then
\begin{equation*}
b(R)\ =\ a(R)+(n-m)d\ =\ -t\Big(\sum_{i=1}^m e_i + \sum_{j=1}^n f_j\Big) + (n-m) \sum_{j=1}^t f_j.
\end{equation*}
The graded anticanonical module of $R$ is given as 
$\omega_R^{-1} = \frakq^{n-m}(-a(R))$, 
where we let $\frakq \colonequals \frakq_{t}(X)$.
If $R$ is standard-graded, then $a(R)=-nt$, while $b(R)=-mt$.

\item Let $X$ be a symmetric $n\times n$ matrix of indeterminates, and $R\colonequals k[X]/I_{t+1}(X)$, where we assume that $2\le t+1 \le n$.
Let $e_1 \le \cdots \le e_n$ be positive integers either all odd or all even, and set $\deg X_{ij} \colonequals (e_i + e_j)/2$.
Let $d$ be the least degree of a $t$-minor and let $\frakp \colonequals \frakp_{t}(X)$. 
We have
\begin{equation*} 
	\omega_R^{-1} \ = \
	\begin{cases}
		R(-a(R)) & \text{ if }  n\not\equiv t\mod 2, \\
		\frakp(-a(R)) & \text{ if }  n\equiv t\mod 2.
	\end{cases}
\end{equation*}

If $n\not\equiv t\mod 2$, then $b(R)=a(R)=-(t/2)\sum_{i=1}^n e_i$;
if $R$ is standard-graded, then this value is $-tn/2$.

If $n\equiv t\mod 2$, then $b(R)=a(R)+d=-(t/2)\sum_{i=1}^n e_i + d/2$;
if $R$ is standard-graded, then $a(R)=-t(n+1)/2$, while $b(R)=-t(n-1)/2$.

\item Let $X$ be an $m \times n$ Hankel matrix of indeterminates, and $R\colonequals k[X]/I_{t+1}(X)$ be standard-graded, where we assume that $2\le t+1 \le m \le n$. 

Then $a(R)=-t$; this agrees with $b(R)$ if $t+1=m=n$, else $b(R)=0$.

If $t+1=m$, then we have $\omega_R^{-1} = \frakq^{(n-t-1)}(t)$, where $\frakq \colonequals \frakq_t(X)$.

\item Let $X$ be an alternating $n\times n$ matrix of indeterminates, and $R\colonequals k[X]/\Pf_{2t+2}(X)$, where we assume that $2\le 2t+2 \le n$.
Let $e_1,\dots,e_n$ be positive integers either all odd or all even, and set $\deg X_{ij} \colonequals (e_i + e_j)/2$.

We have $\omega_R^{-1} = R(-a(R))$, and
$b(R)=a(R)=-t\sum_{i=1}^n e_i$; if $R$ is standard-graded, then this value is $-tn$.

\item Let $X$ be an $m \times n$ matrix of indeterminates and $R$ the subalgebra of the standard-graded polynomial ring $k[X]$ generated by the maximal minors of $X$, where we assume $m \le n$.
If $m=n$, then $b(R)=a(R)=-m$, else we have $b(R)=a(R)=-mn$; 
in either case, we have $\omega_R^{-1} = R(-a(R))$.
\end{enumerate}
\end{theorem}

We first recall some basic facts.

\begin{lemma}
Let $\frakp$ and $\frakq$ be distinct height one prime ideals of a normal domain $R$.
\begin{enumerate}[\quad\rm(1)]
\item If $\frakp \cap \frakq = (x)$ for an element $x\in R$,
then $(\frakp^{(n)})^{-1} = \frac{1}{x^n} \frakq^{(n)}$ for all $n \ge 0$.
\item If $\frakp^{(m + n)} = (x)$ for some $m, n \ge 0$ and $x\in R$, then $(\frakp^{(m)})^{-1} = \frac{1}{x} \frakp^{(n)}$.
\end{enumerate}
\end{lemma}

\begin{proof}
As all ideals involved are divisorial, the assertions may be checked after localizing at height one prime ideals, i.e., over a DVR.
\end{proof}

\begin{lemma}
For graded modules $M$ and $N$ over a graded ring $R$ and an integer $n$, one has degree-preserving isomorphisms
\begin{equation*}
\pushQED{\qed}
\Hom_R(M(-n), N)\ \cong\ \Hom_R(M, N)(n)\ \cong\ \Hom_R(M, N(n)).
\qedhere \popQED
\end{equation*}
\end{lemma}

\begin{proof}[Proof of Theorem~\ref{theorem:determinantal}]
(1) Let $\frakp$ and $\frakq$ respectively denote the prime ideals $\frakp_t(X)$ and~$\frakq_t(X)$,
and let $\delta$ denote the upper left $t$-minor of $X$.
Then $\deg \delta = d$, and $\delta$ is an element of minimal degree in $\frakp$ and in $\frakq$.
The $a$-invariant and the graded canonical module of $R$ are computed in \cite[Corollaries~1.5 and~1.6]{BH:ainv}.
We~obtain
\begin{equation*}
\omega_R \ =\ \frakp^{(n-m)}(-s),
\end{equation*}
where $s=-a(R)-d(n-m)$.
The symbolic and ordinary powers of $\frakq$ coincide, see \cite[Corollary~7.10]{BV}, and we have $\frakp \cap \frakq = (\delta)$, giving us
\begin{equation*}
\omega_R^{-1}\ =\ \frac{1}{\delta^{n-m}}\frakq^{n-m}(s).
\end{equation*}
Thus, the minimal degree of a nonzero element in $\omega_R^{-1}$ equals $-s=a(R)+d(n-m)$.

(2) The ring~$R$ has class group $\ZZ/2$ by \cite{Goto1, Goto2}, and is Gorenstein precisely when one has $n\not\equiv t\mod 2$.
Let $\frakp$ denote the ideal $\frakp_t(X)$;
then $\frakp$ is a height one prime by~\cite[Theorem~1]{Kutz}, and generates the divisor class group of~$R$,
with $\frakp^{(2)}$ being generated by the determinant of the upper left $t\times t$ submatrix of $X$, see~\cite{Goto1}.
It follows that the minimal degree of an element in $\frakp$ is $d$, and that
\begin{equation*}
\frakp^{-1}\ =\ \frakp(d).
\end{equation*}
As $\omega_R$ must be isomorphic to $\frakp$ in the non-Gorenstein case, we obtain
\begin{equation*}
\omega_R\ =\
\begin{cases}
\frakp\left(a(R)+d\right) & \text{ if } n\equiv t\mod 2,\\
R\left(a(R)\right) & \text{ if } n\not\equiv t\mod 2.
\end{cases}
\end{equation*}
It follows that $a(R)=b(R)$ if $n\not\equiv t\mod 2$;
this is the case where the ring $R$ is Gorenstein. If, instead, $n\equiv t\mod 2$, then one has
\begin{equation*}
\omega_R^{-1}\ =\ \frakp^{-1}(-a(R)-d)\ =\ \frakp(-a(R)),
\end{equation*}
giving us the minimal degree of a nonzero element in $\omega_R^{-1}$ as $a(R)+d$.
The calculation of the $a$-invariant may now be found in \cite[Theorem~4.4]{Barile:a:inv} or~\cite[Theorem~2.4]{Conca:symmetric:ladders}.

(3) The ideal $I_{t+1}(X)$ remains unchanged if we rearrange the indeterminates so as to form a $(t+1) \times N$ Hankel matrix $X$, where $N \colonequals n+m-t-1$;
c'est un r\'esultat de Gruson et Peskine \cite[Lemme~2.3]{Gruson-Peskine}.
Note that $t+1=N$ holds precisely when $t+1=n=m$.
Set~$\frakq$ to be the ideal $\frakq_t(X)$.
Summarizing the relevant results from \cite[Theorem~3.1]{CMSV},
we have $\omega_R \cong \frakq^{(2)}$ as graded modules,
$\frakq^{(N-t+1)} = (\delta)$ with $\deg \delta = t$,
and $\frakq^{(k)}$ is generated in degree $t$ for all $1 \le k \le N-t+1$.

When $t+1=N$, i.e., $R$ is a hypersurface, this recovers the obvious, namely $\omega_R=R(-t)$. When $t+1<N$, one has
\begin{equation*}
\omega_R^{-1}\ =\ \frakq^{(N-t-1)}(t),
\end{equation*}
so $b(R)=0$.

(4) and (5) The respective $a$-invariants are computed in \cite[Corollaries~1.7 and~1.4]{BH:ainv}, where we note that we need to scale the grading by $m$ in the latter case.
In view of Remark~\ref{remark:UFD-a-equals-b}, it suffices to note that $R$ is a UFD in each case;
see \cite{Avramov} and \cite[Proposition~8.5]{Samuel} for the respective cases.
\end{proof}

\appendix
\section{Inverses of semiinvariants}

We record some results that are possibly known to experts, but for which we could not find the desired form in the literature.
As in Section~\ref{section:invariants}, fix a field $k$, and a finite subgroup~$G$ of $\GL_n(k)$ with its natural action on $S \colonequals k[x_1,\dots,x_n]$.
Let $\mu \colon G \to k^{\times}$ be a character.
As~$S$ is a UFD, we may define
\begin{equation*}
f_{\mu} \colonequals \gcd(S^G_{\mu}) \in S,
\end{equation*}
which is well-defined up to a unit.
Because $S^G_{\mu}$ is $G$-stable, the element $f_{\mu}$ is a semiinvariant for some eigencharacter $\kappa_{\mu}$.
Setting $\lambda_{\mu} \colonequals \mu/\kappa_{\mu}$, it is a straightforward check that
\begin{equation}
\label{equation:factor-out-gcd}
S^G_{\mu}\ =\ f_{\mu} S^G_{\lambda_{\mu}}.
\end{equation}

\begin{theorem}
\label{theorem:inverses-semiinvariants}
Let $k$ be a field, $G$ a finite subgroup of $\GL_n(k)$ acting on $S \colonequals k[x_1,\dots,x_n]$,
and $\mu \colon G \to k^{\times}$ a character.
Let $\lambda \colonequals \lambda_{\mu}$ and $f_{\mu}$ be as defined above.
The character $\lambda$ is trivial on pseudoreflections.
As graded $S^G$-modules, we have
\begin{equation*}
{(S^G_{\lambda})}^{-1}\ =\ S^G_{\lambda^{-1}},
\quad \text{and} \quad
{(S^G_{\mu})}^{-1}\ \cong\ S^G_{\lambda^{-1}}(\deg f_{\mu}).
\end{equation*}
If $\mu$ is trivial on pseudoreflections, then $\lambda = \mu$.
\end{theorem}

\begin{proof}
Let $g$ be a pseudoreflection and fix a linear form $\ell \in S$ such that $(1 - g)(S) \subseteq \ell S$.
For $s \in S^G_{\lambda}$, we thus get $(1 - \lambda(g))s \in \ell S$.
If $\lambda(g) \neq 1$, then $\ell$ divides each element of $S^G_{\lambda}$, contradicting the fact that $\gcd(S^G_{\lambda})=1$.
Thus, $\lambda$ is trivial on pseudoreflections.

We prove ${(S^G_{\lambda})}^{-1}\subseteq S^G_{\lambda^{-1}}$; the other containment holds in general.
Given $f \in {(S^G_{\lambda})}^{-1}$, it is clear that $f \in L^G_{\lambda^{-1}}$, so we only need to show that $f \in S$.
Writing $f = a/b$ with $a$ and $b$ coprime, we note that
\begin{equation*}
a S^G_{\lambda}\ \subseteq\ b S^G\ \subseteq\ b S,
\end{equation*}
so $b$ divides each element of $S^G_{\lambda}$.
As $\gcd(S^G_{\lambda})=1$, we conclude that $b$ is a unit.
The isomorphism in the display is now a consequence of~\eqref{equation:factor-out-gcd}.

The last statement follows from the description of $f_{\mu}$ in Remark~\ref{remark:description-theta-f} below.
\end{proof}

\begin{remark}
\label{remark:interpret-inverse}
Here is how one can interpret the theorem:
the na\"ive equality ${(S^G_{\mu})}^{-1} = S^G_{\mu^{-1}}$ holds when the gcd, namely $f_{\mu}$, is a unit, equivalently, when $\mu$ is trivial on pseudoreflections.
Once we factor out the gcd, we can write $S^G_{\mu} = f_{\mu} S^G_{\lambda}$ with $\lambda$ trivial on pseudoreflections and obtain
${(S^G_{\mu})}^{-1} = f_{\mu}^{-1} S^G_{\lambda^{-1}}$.
\end{remark}

\begin{remark}
\label{remark:description-theta-f}
We recall the descriptions of $\theta$ and $f_{\mu}$ from \cite[Proposition~9]{Broer} and \cite[Proposition~2.7]{Nakajima}:
Let $W_1,\dots,W_s$ be the distinct hyperplanes that arise as fixed points of elements of $G$.
For each such hyperplane $W_i$, fix a nonzero linear form $\ell_i \in S$ that vanishes on $W_i$.
Let $p$ be the characteristic of $k$ if this is positive, else let $p = 1$.
Let $H_i \subseteq G$ be the pointwise stabilizer of $W_i$, and write $\md{H_i} \colonequals e_i p^{a_i}$ with $e_i$ coprime to $p$.

Let $P_i$ be the subgroup of $H_i$ generated by the transvections.
Define $q_i \colonequals -a(S^{P_i}) - n$,
i.e., $n + q_i$ is the sum of the degrees of the generators of the polynomial ring $S^{P_i}$.
Then $\theta$ is the product $\prod \ell_i^{e_i + q_i - 1}$.

If $\mu \colon G \to k^{\times}$ is a character, and $n_i$ the smallest nonnegative integer such that the character $\mu \det^{n_i}$ is trivial on $H_i$,
then $f_{\mu}$ is the product $\prod \ell_i^{n_i}$.
This formulation may also be found in \cite[\S3]{Broer}, though it is not explicitly stated that $f_{\mu}$ coincides with the gcd.
\end{remark}

\begin{lemma}
\label{lemma:det-chi-pseudo}
Let $\chi$ be the eigencharacter of $\theta$ as earlier.
The character $\det/\chi$ is trivial on pseudoreflections.
\end{lemma}

\begin{proof}
Since a transvection has order $p$, each character takes value $1$ on a transvection.
Suppose $g$ is a diagonalizable pseudoreflection, set $H \colonequals \langle g \rangle$.
After a change of coordinates, we may assume that $g$ is of the form $\diag(\zeta, 1, \dots, 1)$ with $\zeta$ a primitive $d$-th root of unity.
Then $S^H = k[x_1^d, x_2,\dots,x_n]$, with the corresponding Dedekind different $\calD_{S/S^H}$ generated by
$\theta_H = x_1^{d-1}$.
As $S^H$ is a UFD, the Dedekind different $\calD_{S^H/S^G}$ is a principal ideal;
fix a generator $\theta_{G/H} \in S^H$.
Transitivity of the different \cite[Lemma~3.10.1]{Benson} yields
$\theta = \theta_H \cdot \theta_{G/H}$.
Using this, we conclude that
\begin{equation*}
\chi(g)\ =\ \frac{g(\theta)}{\theta}\ =\ \frac{g(\theta_H)}{\theta_H}\ =\ \frac{1}{\zeta^{d - 1}}\ =\ \zeta\ =\ \det(g).
\qedhere
\end{equation*}
\end{proof}

\section*{Acknowledgments}

Some of the examples were verified using the computer algebra system \texttt{Magma} \cite{Magma}.
We thank Akiyoshi Sannai for pointing out an error regarding regrading in Theorem~\ref{theorem:determinantal}.


\end{document}